\documentclass[reqno]{llncs}
\usepackage{amsmath}
\usepackage{amsfonts}
\usepackage{amssymb}
\usepackage{blkarray}
\usepackage{mathtools}
\usepackage{multirow}
\usepackage{bbold}

\newcommand{\R}{\mathbb{R}}

\newcommand{\supp}{\text{supp}}
\newcommand{\spc}{\text{ }}

\newcommand{\y}{\textbf{y}}
\newcommand{\x}{\textbf{x}}
\newcommand{\p}{\textbf{p}}
\newcommand{\w}{\textbf{w}}

\usepackage{algorithm}
\usepackage[noend]{algpseudocode}
\usepackage{tikz}

\usepackage[margin=1.5in]{geometry}

\usepackage[colorlinks=true,breaklinks=true,bookmarks=false,urlcolor=blue,citecolor=blue,linkcolor=blue,bookmarksopen=false,draft=false]{hyperref}

\begin{document}

\title{A Column Generation Approach to the\\ Discrete Barycenter Problem }

\author{Steffen Borgwardt\inst{1} \and Stephan Patterson\inst{2}}

\institute{\email{\href{mailto:steffen.borgwardt@ucdenver.edu}{steffen.borgwardt@ucdenver.edu}};
University of Colorado Denver 
\and \email{\href{mailto:stephan.patterson@lsus.edu}{stephan.patterson@lsus.edu}}; Louisiana State University in Shreveport
}

\date{}

\maketitle
\begin{abstract}
The discrete Wasserstein barycenter problem is a minimum-cost mass transport problem for a set of discrete probability measures. Although an exact barycenter is computable through linear programming, the underlying linear program can be extremely large. For worst-case input, a best known linear programming formulation is exponential in the number of variables, but has a low number of constraints, making it an interesting candidate for column generation.

In this paper, we devise and study two column generation strategies: a natural one based on a simplified computation of reduced costs, and one through a Dantzig-Wolfe decomposition. For the latter, we produce efficiently solvable subproblems, namely, a pricing problem in the form of a classical transportation problem. The two strategies begin with an efficient computation of an initial feasible solution. While the structure of the constraints leads to the computation of the reduced costs of all remaining variables for setup, both approaches may outperform a computation using the full program in speed, and dramatically so in memory requirement. In our computational experiments, we exhibit that, depending on the input, either strategy can become a best choice. \end{abstract}
 

\section{Introduction}\label{sec:intro}

Optimal transport problems involving joint transport to a set of probability measures appear in a variety of fields, including recent work in image processing  \cite{jbtcg-19,sa-19}, machine learning \cite{hbccp-19,shbmccps-18,ydpbbgc-19}, and graph theory \cite{uddgn-18}, to name but a few. The so-called {\em Wasserstein barycenters} take a central role in many of these applications: a {\em barycenter} is another probability measure which minimizes the total distance to all input measures (i.e, images) and acts as an average distribution in the probability space; the {\em squared Wasserstein distance} is of particular consideration due to the preservation of the geometric structure of the input. The wide scope of optimal transport problems makes challenging even a reasonably comprehensive summary; for recent monographs on the Wasserstein distance and computational optimal transport, we refer the reader to  \cite{pc-18} and \cite{pz-19}, in addition to the seminal work of Villani \cite{v-09}. 

In almost all applications, the probability measures have discrete support, i.e., a finite number of points to which positive mass is associated. This leads to the so-called {\em discrete barycenter problem}, defined as follows: Given a set of probability measures $P_1, \ldots, P_n$, each with a finite set $\supp{(P_i)}$ of support points in $\R^d$ and associated masses, and a set of $n$ nonnegative weights $\lambda_i \in \R$ with $\sum_{i=1}^n \lambda_i =1$, find a probability measure $\bar P$ on $\R^d$, that is, a (Wasserstein) barycenter, satisfying
\begin{equation}\label{eqn:barycost}
\phi(\bar P):=\sum\limits_{i=1}^n \lambda_i W_2(\bar P, P_i)^2 = \inf\limits_{P\in \mathcal{P}^2(\R^d)} \sum\limits_{i=1}^n \lambda_i W_2(P,P_i)^2,
\end{equation}
where $W_2$ is the quadratic Wasserstein distance and $\mathcal{P}^2(\R^d)$ is the set of all probability measures on $\R^d$ with finite second moments \cite{ac-11}. 
Since $P_1, \ldots, P_n$ have finite sets of support points, we call the measures {\em discrete}. As the $P_i$ are measures, the total mass of their support points sums up to $1$. With $d_i$ representing the mass of $\x_i \in \supp{(P_i)}$, this can be denoted as $\sum_{\x_i\in\supp{(P_i)}} d_{i} = 1$. Because $P_1, \ldots P_n$ are discrete, the solution measure $\bar P$ also has a finite set of support points, and the Wasserstein distance is the squared Euclidean distance \cite{abm-16,coo-15,v-09}.

The discrete barycenter problem is a multi-marginal optimal transport problem, and as such is significantly more challenging than the classical two-marginal optimal transport problem referenced later in Section \ref{sec:price}. In fact, multi-marginal optimal transport is no longer a network flow problem \cite{lhcj-19}, and a variant of the problem -- finding optimal solutions with a bound on the size of the support set -- has recently been shown to be NP-hard \cite{bp-21a}. Further, there is current work on NP-hardness in all situations \cite{ab-21}.

Therefore, considerable activity continues on exact, approximate and heuristic methods of computation, such as alternating minimization algorithms \cite{qp-18,tdgu-20,ylst-19}. State-of-the art approximation methods solve the entropy regularized optimal mass transport problem introduced in \cite{c-13}. Entropic regularization leads to a strongly convex program, and its smoothing effects give qualitatively different solutions from exact barycenters \cite{bccnp-14}. Entropy regularized transport problems can be solved efficiently, with a linear-in-$n$ complexity bound, using iterative Bregman projection algorithms \cite{bccnp-14,cd-14,twk-18}, although work continues on the stability and complexity of these algorithms, e.g., \cite{kddgtu-19}. In contrast, exact solutions to the discrete barycenter problem are commonly used for benchmarking purposes and only possible for small input; one of their advantages is that they find a barycenter of provably sparse support and associated transport that is non-mass splitting (see Definition \ref{def:nonmasssplit}). Exact barycenters can be computed by linear programming \cite{ac-11,abm-16,coo-15}, but all known LP formulations scale exponentially. 

The vast majority of algorithms in the literature, including the above examples,  are based on an explicit specification of a discrete set $S \subset \mathbb{R}^d$ of support points that may be allocated mass; see, e.g., \cite{bccnp-14,b-17,bp-18,coo-15,cd-14,kddgtu-19,scsj-17}. The search for an optimal $P \in \mathcal{P}^2(\R^d)$ in Eq.~\ref{eqn:barycost} is replaced by a search over $\mathcal{P}^2(S)$. The size of $S$ typically is the main bottleneck for the practical performance of algorithms \cite{ab-21}. 

Different types of input lead to a different level of challenge. In image processing, for example, the probability measures are supported on the same structured set (a pixel grid). This highly structured support is a best-case input. In this setting, a barycenter can be computed exactly in polynomial time \cite{bp-18}. In practice, the cost is still prohibitive: an exact barycenter lies in an $n$-times finer grid. It is common practice to use a coarser grid to find an approximate barycenter. The original grid already contains a $2$-approximation \cite{b-17}.

By contrast, a worst-case input occurs for measures with no known structure, such as in wildfire ignition points or crime locations \cite{bp-18}. Then it becomes difficult to specify a small set $S$ of possible support points for a `good' approximation or one that allows for the computation of an exact barycenter. However, the existence of sparse solutions to the problem \cite{abm-16} for any input indicates that strategies which dynamically introduce support points, or collections of support points, would be promising to approach these difficult instances.

In this paper, we use linear programming theory and column generation techniques to take a step in this direction. We will advance the state-of-the-art on the computation of exact barycenters for such worst-case input. These computations will still remain costly.

\subsection{Linear Programming for The Discrete Barycenter Problem}\label{sec:LPforBary}

The discrete barycenter problem can be solved exactly by linear programming \cite{ac-11,abm-16,bp-18,coo-15}, but all known LP formulations may require an exponential number of variables, scaling by the product of the sizes of the support sets of the input measures \cite{bp-18}. Some formulations also have an exponential number of constraints, but for any input there exists one with an extremely low number of constraints. These dimensions indicate the linear program is a promising candidate for column generation. 

The so-called {\em non-mass-splitting property}, satisfied by all exact barycenters, is a crucial tool for the linear programming approach to the problem that we use in this paper. It states that any optimal transport plan (the support points in each $P_1, \ldots, P_n$ to which each support point in $P$ transports mass, and the amount of mass transported) of a barycenter may not send mass to more than one support point in each measure \cite{ac-11,abm-16}. This property is fundamental to the modeling of many physical applications where a mass split would be infeasible.

\begin{definition}[The Non-mass-splitting Property]\label{def:nonmasssplit} The mass of each barycenter support point is transported fully to a single support point in each measure; that is, for each $x_k \in \supp(\bar P)$ with corresponding mass $z_k$, $k = 1, \ldots, |\supp(\bar P)|$, there exists exactly one $x_{i} \in \supp(P_i)$, $i = 1, \ldots, n$ to which the entire mass $z_k$ is transported in any optimal transport plan.  \end{definition}

Since the non-mass-splitting property holds for all barycenters, each support point in a barycenter is associated with a single combination of input support points, consisting of the points to which its mass is transported. The set of combinations of input support points is denoted $S^* =  \{ (\x_{1}, \ldots, \x_{n}) : \x_{i} \in \supp(P_i) \text{ for } i = 1, \ldots, n \}$, with elements $s_h = (\x_1^h, \x_2^h, \ldots, \x_n^h)$, $h = 1,\ldots, |S^*|$. Each combination $s_h$ has an associated \emph{weighted mean} $\x^h = \sum_{i=1}^n \lambda_{i} \x_{i}^h$. The weighted mean $\x^h$ is the optimal location for joint mass transport to the points in the combination $s_h$. Therefore, specifying the set $S$ to contain all distinct weighted means makes it the set of {\em all possible support points} for the barycenter.

This notation allows us to formally describe the worst-case setting to which our algorithm will be tailored: when each combination $s_h$ produces a different weighted mean $\x^h$. Then we say the measures $P_1, \ldots, P_n$ are in \textit{general position}, and using $|P_i|$ to denote the size of the support set of $P_i$, the number of distinct $\x^h$ is $|S^*| = \prod_{i=1}^n |P_i|$. Thus the number of weighted means is exponential in the number of input measures $n$ -- without additional knowledge, the set of possible support points $S$ would be of size $|S|=|S^*|$.

We provide an example of a discrete measure in $\R^2$ in Figure \ref{fig:measures} (left), and three measures in general position in Figure \ref{fig:measures} (right). For this tiny example, verifying that the measures are in general position is elementary but somewhat tedious, as the weighted means of all 27 combinations of support points must be computed and verified as unique; in general, verifying whether a particular set contains the correct possible support points is NP-hard \cite{bp-18}. A barycenter for these measures is displayed in Figure \ref{fig:measwithbary} (left), shown with associated transport in Figure \ref{fig:measwithbary} (right).

\begin{figure}[!t]
\begin{center}
\fbox{\includegraphics[scale=.5]{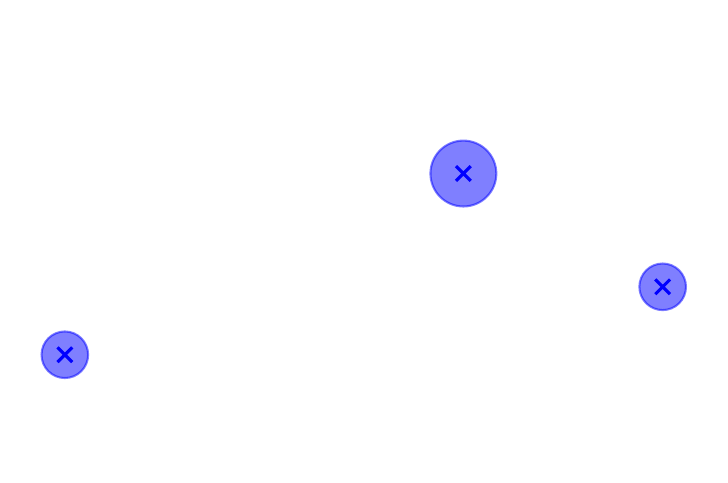}}
\fbox{\includegraphics[scale=.5]{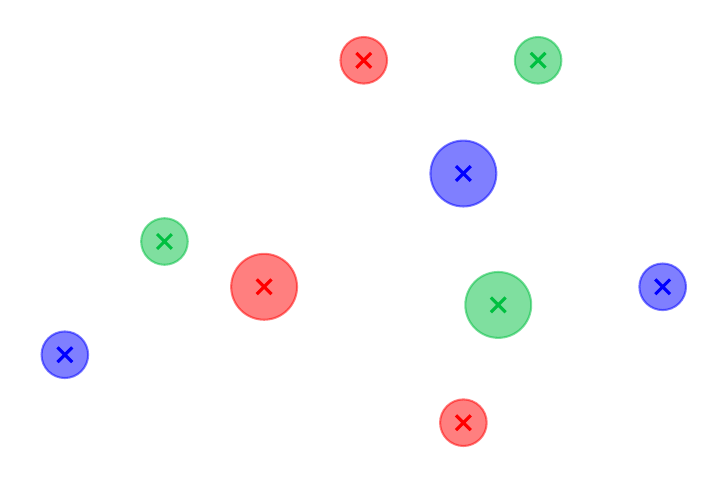}}
\end{center}
\caption{(left) A discrete probability measure  in $\R^2$ with three support points. The size of the points indicates their associated mass. (right) Three measures in $\R^2$ each with three support points. All combinations $(\x_1, \x_2, \x_3)$ produce a different weighted mean; therefore these measures are in general position.}\label{fig:measures}
\end{figure}

\begin{figure}
\begin{center}
\fbox{\includegraphics[scale=.5]{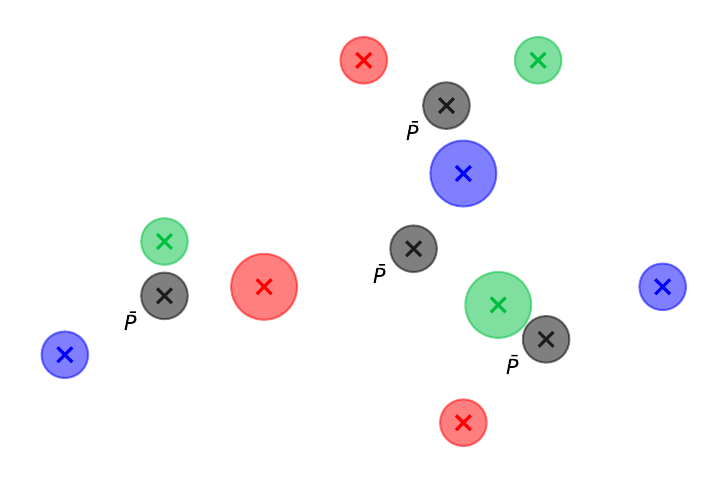}}
\fbox{\includegraphics[scale=.5]{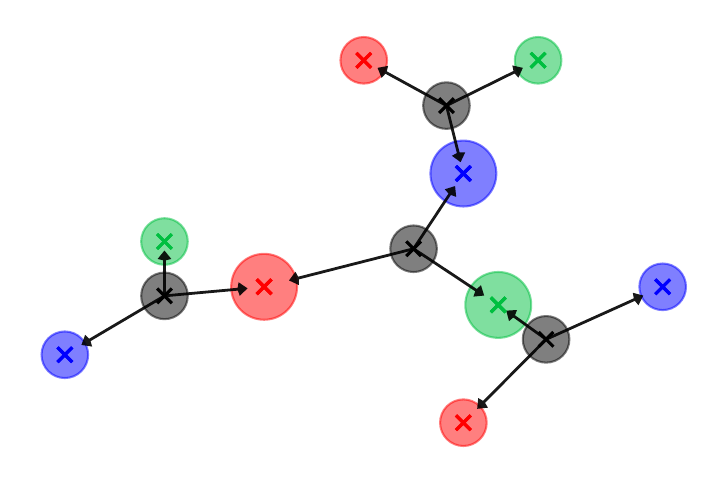}}
\end{center}
\caption{(left) Assuming $\lambda_i=\frac{1}{3}$ for $i = 1,2,3$, a barycenter $\bar P$ for the three measures from Figure \ref{fig:measures}. Each support point has mass $\frac{1}{4}$.  (right) The mass transport from each barycenter support point to the original measures. Each barycenter support point is the weighted mean of the points to which it transports.}\label{fig:measwithbary}
\end{figure}

In this paper, we build on an LP formulation from \cite{bp-18} that, among known formulations, requires the fewest variables and constraints for general position measures. 
In this formulation, which we call LP (\ref{LPw}), a variable is introduced for each combination $s_h$ of support points from $S^*$: each $s_h$ has a corresponding variable $w_h$ representing the mass assigned to $\x^h$ and transported fully to each $\x_i^h$, $i = 1, \ldots, n$. The total transport cost of a unit of mass from $\x^h$ is given by $c_h = \sum_{i=1}^n ||\x^h-\x_{i}^h||^2$. 

Constraints arise from the requirement that the total transport to each support point $\x_i$ in each measure is exactly equal to its mass $d_i$. This produces one equality constraint for each $\x_i$ in each measure; that is, 
\[ \sum_{h:\x_i^h = \x_i} w_h = d_i, \spc\spc\spc \forall i = 1, \ldots, n, \spc\spc \forall \x_i \in \supp(P_i).  \]

Recall that $\sum_{\x_i\in\supp{(P_i)}} d_{i} = 1$ for each $P_i$. Thus there exists a feasible solution to the union of all these constraints, and it has to satisfy $\sum_{h=1}^{|S^*|} w_h = 1$. The constraints can be represented as a real $\sum_{i=1}^n |P_i| \times \prod_{i=1}^n |P_i|$ matrix $A$ times a vector $w$, equal to right-hand side $d$. In $A$, column $h$ contains ones in the $n$ rows where $\x_i^h = \x_i$, and zeroes otherwise. The matrix $A$ is highly structured, which will be both good and bad for our purposes; we take a closer look at it in Section \ref{sec:price}. With $w$ as the vector of variable masses $w_h$, $c$ as the vector of associated costs $c_h$, and $d$ as the vector of masses $d_i$ of the support points $\x_i$ in the input, a full formulation of LP (\ref{LPw}) is as follows: 

\begin{equation*}\label{LPw}
\begin{array}{crl}
\tag{general}  \mathrm{min}&  c^T \w & \nonumber \\
\mathrm{s.t.  } &A \w = &d \\
&\w \geq & 0. \nonumber
\end{array}
\end{equation*}

The number of constraints $\sum_{i=1}^n |P_i|$ scales linearly, equal to the total number of support points in the input measures. Meanwhile, the number of variables $\prod_{i=1}^n |P_i|$ scales exponentially in the number of input measures. This extreme difference in the scaling of number of rows and columns further motivates our interest in a column generation approach. 




\subsection{Outline}

In Section \ref{sec:price}, we develop two column generation algorithms and discuss how to exploit the structure of the problem for efficient pricing and a memory-efficient implementation. In Section \ref{sec:master}, we devise a simple greedy algorithm to find an initial feasible solution to start these algorithms. Section \ref{sec:comp} contains some computational experiments. The results demonstrate the practical advantages of the algorithms over direct computations using LP (\ref{LPw}), and exhibit that the different algorithms become a best choice depending on the input.  We finish with some concluding remarks in Section \ref{sec:conc}.



\section{Column Generation}\label{sec:price}

We briefly recall the basics of a column generation strategy for solving a linear program; for additional details see for instance \cite{dl-05}. Column generation is the process of dynamically adding variables to any linear program, and begins with a version of the linear program containing a small subset of the variables. This (reduced) linear program is called the \textit{master problem}. Additional variables are chosen using a sub-problem, which we call the \textit{pricing problem}, in which the current optimum of the restricted master problem is used to produce a new, potentially improving column. Column generation terminates when the optimal value of the pricing problem is no longer negative.

The pricing problem must contain enough information for its solution to be meaningful, but also remain sufficiently simple to be efficiently solvable. In our first column generation algorithm, we apply column generation naturally to LP (\ref{LPw}), described in Section \ref{sec:reducedcosts}. In our second column generation algorithm, we apply column generation to an alternative linear program, described in Section \ref{sec:LPDW}. We produce an efficiently solvable pricing problem for the alternative LP by exploiting the structure of the constraint matrix; this structure is described in Section \ref{sec:coeffstruc}.

\subsection{Column Generation for LP (\ref{LPw})}\label{sec:reducedcosts}

We first devise a strategy to apply column generation directly to LP (\ref{LPw}). At any iteration of the column generation process, the master problem contains variables which correspond to a subset of the possible combinations of support points $S^*$, and because the measures are assumed to be in general position, each variable corresponds to a unique point in the set of all possible support points $S$. Thus an improving column produced by column generation is precisely a new support point to include in $S$, and column generation for LP (\ref{LPw}) corresponds directly to the stated goal of the dynamic generation of $S$. 

We produce a pricing problem as follows. Let $\y$ be the vector containing the dual values associated with the mass transport constraints of any master problem based on LP (\ref{LPw}). Then a pricing problem using the vector of reduced costs $c - \y^T A$ is:
\begin{equation}\label{eqn:reducedcost}
\min_h \{ c_h - \y^TA_h\} 
\end{equation}
where $h = 1, \ldots, \prod_{i=1}^n |P_i|$ and $A_h$ is the $h^{th}$ column of $A$. Since the vectors $\y$ and $A_h$ contain $\sum_{i=1}^n |P_i|$ elements, each individual evaluation of $c_h - \y^T A_h$ is efficient, even if $c_h$ has not yet been computed. For instance, as in \cite{bp-18}, $c_h$ can be computed directly and efficiently from the combination $s_h = (\x_1^h, \ldots, \x_n^h) \in S^*$ using
\begin{equation*}\label{eqn:newcost} c_h = \sum_{i=1}^{n-1} \lambda_i \sum_{k=i+1}^n \lambda_k ||\x_k^h - \x_i^h||^2.\end{equation*} 
Therefore the exponential scaling of $c$ with the number of measures $n$ is the sole source of inefficiency for using Equation (\ref{eqn:reducedcost}) in column generation. The structure of $A$, described in Section \ref{sec:coeffstruc}, allows a memory-efficient evaluation of the product $\y^T A_h$.

Equation (\ref{eqn:reducedcost}) produces {\em one} column to introduce to the master problem each iteration of column generation: the variable with index $h$ where $c_h-\y^T A_h$ is minimal. To reduce the number of times the exponential vector must be processed, in our computational experiments, we also consider two alternative pricing problems: one where we introduce {\em all} columns $h$ where $c_h -y^T A_h$ is negative, and one where we introduce the best $n$ columns to the master problem at each iteration. 

\subsection{Dantzig-Wolfe Reformulation}\label{sec:LPDW}

In this section, we present a linear program based on the convex hull of vertices of a polyhedron. If the polyhedron is generated by the constraints of LP (\ref{LPw}), that is, $\{\w \in \R^d: A\w = d, \w \geq 0\}$, then the linear program is an alternative to LP (\ref{LPw}). Vertex-form linear programs are not generally considered for computation, since the majority of such polyhedra have exponentially many vertices. However, for general position measures, LP (\ref{LPw}) already scales exponentially, so computations using a vertex formulation face the same challenges, and as we will see, have the same potential benefits from column generation with a low number of constraints and large number of variables.

In our second column generation algorithm, in addition to reformulating LP (\ref{LPw}) for a new vertex-form master problem, we also perform a decomposition of the constraints. A decomposition of a vertex-form linear program is called a {\em Dantzig-Wolfe reformulation}, presented in \cite{dw-60}. The decomposition begins by partitioning the constraint matrix as $A = \begin{bmatrix} A_p \\ A_m \\ \end{bmatrix}$ and right-hand side $d = \begin{pmatrix} d_p \\ d_m \\ \end{pmatrix}$. A preselected number of rows are assigned to the matrix $A_p$ for use in the separate pricing problem. Note that reordering the rows of $A\w = d$ does not affect the underlying polytope, so the matrix $A_p$ does not need to be precisely the first rows of $A$; however, in our analysis in Section \ref{sec:coeffstruc} we will assume that the measures have been ordered such that those chosen for the pricing problem are first. The remaining rows of $A$ are assigned to the matrix $A_m$ and remain in the master problem. 

The pricing problem produces vectors $\p$ that are vertices of the polyhedron $\{\p \in \R^d: A_p\p = d_p, \p \geq 0\}$). These vectors represent potential distributions of mass to each possible combination in $S^*$, but are typically not (individually) feasible for the full problem $A\w = d$. They are added to the master problem through the products $c^T \p$ and $A_m \p$ in the objective and constraints, respectively. These products for all produced $\p$ are then combined in a convex combination with weights in the new variable vector $\mu$ to produce fully feasible solutions. The resulting master problem has both a limited number $J$ of variables due to the column generation, and a slightly reduced number of constraints due to the decomposition, and is now called the {\em restricted master problem}. 

\begin{equation*}\label{LPRM}
\begin{array}{crl}
\tag{RM}  \mathrm{min}  & \sum\limits_{j=1}^{J} (c^T\p_j)\mu_j \nonumber \\
\mathrm{s.t.}&\sum\limits_{j=1}^{J} (A_m \p_j) \mu_j  &= d_m \\
&\sum\limits_{j=1}^{J} \mu_j &=  1 \\
&\mu_j  &\geq  0, \forall  j=1,\ldots,J\nonumber
\end{array}
\end{equation*}

We confirm that the structure that makes LP (\ref{LPw}) a prime candidate for column generation is preserved in LP (\ref{LPRM}): the number of constraints is bounded above by $\sum_{i=1}^n |P_i|+1$, as LP (\ref{LPRM}) has just one additional constraint for convexity and a (possibly improper) subset of the rows. Ideally, only a fraction of the total number of vertices are used in LP (\ref{LPRM}), so that the number of columns remains low, as well. 


Recall that the pricing problem uses the current optimum of the restricted master problem to produce a new column to introduce to LP (\ref{LPRM}). Specifically, the objective function of the pricing problem requires the dual solution to LP (\ref{LPRM}), where $\y$ was the dual solution corresponding to the constraints $A\w=d$ in LP (\ref{LPw}). We will now denote the dual solution to $(\ref{LPRM})$ by $(\y_m, \sigma)$, where $\y_m$ contains the dual values associated with the mass transport constraints $A_m \p = d_m$, and $\sigma \in \R$ is the dual value associated with the convexity constraint in LP (\ref{LPRM}). Then the base form of the pricing problem is:
\begin{equation*}\label{LPprice}
\begin{array}{crl}
\tag{price}  \mathrm{min} &(c^T-&\y_m^TA_m) \p - \sigma \nonumber \\
\mathrm{s.t.} & A_p \p  &=  d_p \\
&\p &\geq 0.
\end{array}
\end{equation*}

LP (\ref{LPprice}) is still an exponential-sized linear program: The constraint matrix $A_p$ has an exponential number of columns, as does the matrix $A_m$, and the cost vector $c$ has an exponential number of elements. In fact, LP (\ref{LPprice}) contains the same number of variables as LP (\ref{LPw}). We now develop an improved pricing problem using information specific to the barycenter problem.

\subsection{The Structure of the Coefficient Matrix $A$}\label{sec:coeffstruc} 
Recall that $A$ contains only elements $1$ and $0$: in column $h$, there is a $1$ when $\x_{i}$ is in the tuple $s_h$, that is, $\x^h_i = \x_i$, and $0$ otherwise. In fact, each column contains exactly $n$ nonzero coefficients. The pattern created within the matrix $A$ is displayed in Example \ref{ex:A}: each row has consecutive ones alternating with consecutive zeros. For each measure, the consecutive ones start in the first column for the first constraint in each measure, then start in the second row immediately after the end of the previous consecutive ones, continuing to the last constraint of the measure, forming a block. The width of the block depends on the measure $P_i$ with which the constraints are associated. The number of consecutive ones equals the product of the sizes of the measures with a higher index: the rows of $A$ associated with $P_i$,  $1 \leq i < n$, contain $\prod_{l=i+1}^n |P_l|$ consecutive ones. The block for the final measure is the identity matrix. 

\begin{example}\label{ex:A}

The matrix $A$ for four measures with sizes $|P_1| = |P_3| = 2$ and $|P_2| = |P_4|= 3$ contains blocks of ones and zeros. The width of block structure for particular constraints depends on the index $i$ of the corresponding measure $P_i$. Here there are 36 total columns, and the number of consecutive ones for each measure is 18, 6, 3, and 1, respectively. 
\setcounter{MaxMatrixCols}{50} 
\begin{center}
\begin{tabular}{cc}
$ A=  
\begin{bmatrix}
1 & 1 & 1 & 1 & 1 & 1& 1 & 1 & 1 & 1 & 1 & 1 & 1 & 1 & 1 & 1 & 1 & 1 & 0 & 0 & 0 & 0 & 0 & 0 & 0 & 0 & 0 & 0 & 0 & 0 & 0 & 0 & 0 & 0 &0 &0 \\
0 & 0 & 0 & 0 & 0 & 0 & 0 & 0 & 0 & 0 & 0 & 0 & 0 & 0 & 0 & 0 &0 &0 &1 & 1 & 1 & 1 & 1 & 1& 1 & 1 & 1 & 1 & 1 & 1 & 1 & 1 & 1 & 1 & 1 & 1  \\ \hline
1 & 1 & 1 & 1 & 1 & 1 & 0 & 0 & 0 & 0 & 0 & 0 & 0 & 0 & 0 & 0 & 0 & 0 &1 & 1 & 1 & 1 & 1 & 1 & 0 & 0 & 0 & 0 & 0 & 0 & 0 & 0 & 0 & 0 & 0 & 0 \\ 
0 & 0 & 0 & 0 & 0 & 0 & 1 & 1 & 1 & 1 & 1 & 1& 0 & 0 & 0 & 0 & 0 & 0 & 0 & 0 & 0 & 0 & 0 & 0 & 1 & 1 & 1 & 1 & 1 & 1& 0 & 0 & 0 & 0 & 0 & 0 \\ 
0 & 0 & 0 & 0 & 0 & 0 &0 & 0 & 0 & 0 & 0 & 0 &1 & 1 & 1 & 1 & 1 & 1& 0 & 0 & 0 & 0 & 0 & 0 &0 & 0 & 0 & 0 & 0 & 0 &1 & 1 & 1 & 1 & 1 & 1 \\ \hline
1 & 1 & 1 & 0 & 0 & 0 & 1 & 1 & 1 &0 & 0 & 0 &1 & 1 & 1 &0 & 0 & 0 &1 & 1 & 1 & 0 & 0 & 0 & 1 & 1 & 1 &0 & 0 & 0 &1 & 1 & 1 &0 & 0 & 0  \\
0 & 0 & 0 & 1 & 1 & 1 &0 & 0 & 0 & 1 & 1 & 1 &0 & 0 & 0 & 1 & 1 & 1 & 0 & 0 & 0 & 1 & 1 & 1 &0 & 0 & 0 & 1 & 1 & 1 &0 & 0 & 0 & 1 & 1 & 1 \\ \hline
1 & 0 & 0 & 1 & 0 & 0 & 1& 0 & 0 & 1 & 0 & 0 & 1& 0 & 0 & 1 & 0 & 0 & 1 & 0 & 0 & 1 & 0 & 0 & 1& 0 & 0 & 1 & 0 & 0 & 1& 0 & 0 & 1 & 0 & 0 \\
0 & 1 & 0 & 0 & 1 & 0 & 0 & 1 & 0 & 0 & 1 & 0& 0 & 1 & 0& 0 & 1 & 0 & 0 & 1 & 0 & 0 & 1 & 0 & 0 & 1 & 0 & 0 & 1 & 0& 0 & 1 & 0& 0 & 1 & 0  \\
0 & 0 & 1 & 0 & 0 & 1 & 0 & 0 & 1 & 0 & 0 & 1& 0 & 0 & 1& 0 & 0 & 1 & 0 & 0 & 1 & 0 & 0 & 1 & 0 & 0 & 1 & 0 & 0 & 1& 0 & 0 & 1& 0 & 0 & 1 \\
\end{bmatrix}$ & $\begin{matrix}\multirow{2}{*}{$P_1$} \\ \\ \multirow{3}{*}{$P_2$}  \\ \\ \\ \multirow{2}{*}{$P_3$} \\ \\ \multirow{3}{*}{$P_4$} \\ \\ \\ \end{matrix}$
\end{tabular} \\  
\end{center} \qed
\end{example}

Since the number of consecutive ones for each row of $A$ can be generated from the sizes of the support sets $|P_i|$, the columns of $A$ are easily generated solely from the problem input. Formulas for generating a given column $h$ are provided in Algorithm \ref{alg:ACol}. While Algorithm \ref{alg:ACol} is written with the sums and products in their respective formulas, recalculating these values on repeated runs can be avoided by storing the number of consecutive ones and the width of a full block for each measure.

Because $A$ is easily generated, the matrix $A_m$ is {\em not required in memory}. Instead, in the objective function, $\y_m^TA_m$ is calculated using the $|P_i|$ to determine which dual values should be added.  For further computational efficiency, updates to elements of $(c^T-\y_m^T A_m)$ are only required for those values of $\y_m$ which have changed from the previous iteration; due to the sparse nature of $A_m$, many elements may remain unchanged. 

\begin{algorithm}[h]
\caption{Generation of Column $h$}\label{alg:ACol}
\begin{algorithmic}[1]
\Statex Input:\begin{itemize} \item Column Index $h$, assuming the index of the first column is $0$  \item $|P_i|$ for $i = 1, \ldots, n$  \end{itemize}
\Statex Output: Column $h$ of matrix $A$, denoted $A_h$
\State Let $A_h$ be a column of zeros with length $(\sum_{i=1}^n |P_i|)$
\State $j = \lfloor \frac{h}{\prod_{l=2}^n |P_l|} \rfloor$
\State $A_h(j) = 1$
\For {$i=2, \ldots, n-1$}
\State$j = \sum_{l =1}^{i-1} |P_l| + \lfloor \frac{h \pmod{ \prod_{l=i}^n |P_l|}}{\prod_{l=i+1}^n |P_l|} \rfloor$
\State $A_h(j) = 1$
\EndFor
\State $j = \sum_{l =1}^{n-1} |P_l| + h \pmod{|P_n|}$
\State $A_h(j) = 1$
\end{algorithmic}
\end{algorithm}

By taking $A_p$ as the first rows of $A$, the pattern of consecutive ones also guarantees that $A_p$ always has {\em many duplicate columns.} Continuing with the matrix $A$ from Example \ref{ex:A}, in Example \ref{ex:A2}, we assign the constraints for the first two measures to $A_p$, resulting in a matrix with six unique columns, each repeated six times. For any number of measures $n$, partitioning the constraints for $k$ measures, $1 \leq k < n$, to $A_p$ results in  $n_u = \prod_{i=1}^k |P_i|$ unique columns, while the number of times each column is duplicated is $n_d = \prod_{i=k+1}^n |P_i|$. For a fixed $k$, the number of unique columns is no longer exponential; we justify the choice $k=2$ momentarily.  

\begin{example}\label{ex:A2}

Using the matrix $A$ from Example \ref{ex:A}, a decomposition of all constraints associated with the first two measures into the pricing problem gives this matrix $A_p$.

\begin{center}
\begin{tabular}{cc}
$A_p = \begin{bmatrix}
1 & 1 & 1 & 1 & 1 & 1& 1 & 1 & 1 & 1 & 1 & 1 & 1 & 1 & 1 & 1 & 1 & 1 & 0 & 0 & 0 & 0 & 0 & 0 & 0 & 0 & 0 & 0 & 0 & 0 & 0 & 0 & 0 & 0 &0 &0 \\
0 & 0 & 0 & 0 & 0 & 0 & 0 & 0 & 0 & 0 & 0 & 0 & 0 & 0 & 0 & 0 &0 &0 &1 & 1 & 1 & 1 & 1 & 1& 1 & 1 & 1 & 1 & 1 & 1 & 1 & 1 & 1 & 1 & 1 & 1  \\  \hline
1 & 1 & 1 & 1 & 1 & 1 & 0 & 0 & 0 & 0 & 0 & 0 & 0 & 0 & 0 & 0 & 0 & 0 &1 & 1 & 1 & 1 & 1 & 1 & 0 & 0 & 0 & 0 & 0 & 0 & 0 & 0 & 0 & 0 & 0 & 0 \\ 
0 & 0 & 0 & 0 & 0 & 0 & 1 & 1 & 1 & 1 & 1 & 1& 0 & 0 & 0 & 0 & 0 & 0 & 0 & 0 & 0 & 0 & 0 & 0 & 1 & 1 & 1 & 1 & 1 & 1& 0 & 0 & 0 & 0 & 0 & 0 \\ 
0 & 0 & 0 & 0 & 0 & 0 &0 & 0 & 0 & 0 & 0 & 0 &1 & 1 & 1 & 1 & 1 & 1& 0 & 0 & 0 & 0 & 0 & 0 &0 & 0 & 0 & 0 & 0 & 0 &1 & 1 & 1 & 1 & 1 & 1 \\ 
\end{bmatrix}$ & $\begin{matrix} \multirow{2}{*}{$P_1$} \\ \\ \multirow{3}{*}{$P_2$} \\ \\ \\ \end{matrix}$
\end{tabular}
\end{center}

Every column is repeated six times: $|P_3|\cdot|P_4|$. The matrix of unique columns is $U_p$.
\[U_p = 
\begin{bmatrix}
1 &1  & 1 & 0 & 0 & 0 \\
0 & 0 & 0 &1 & 1 & 1   \\ 
1 &  0&  0 &1  & 0 & 0 \\ 
 0 & 1 & 0 & 0 & 1 & 0 \\
0&  0 & 1 & 0 & 0 & 1  \\
\end{bmatrix} 
\]
\end{example}\qed

Replacing the constraint matrix $A_p$ in LP (\ref{LPprice}) with the matrix of unique columns $U_p$ requires a corresponding change to the objective function. Noting that LP (\ref{LPprice}) is the minimization of a linear objective, only the most negative coefficient for each unique column is required: in an optimal solution, all mass is assigned to such a column. Thus, it suffices to keep a best-cost vector $b$ for the unique columns. These two substitutions produce LP (\ref{Uprice}). 

\begin{equation*}\label{Uprice}
\begin{array}{crl}
\tag{Uprice}  \mathrm{min} & b^T \textbf{q}& + \sigma \nonumber \\
\mathrm{s.t.} & U_p \textbf{q}  &=  d_p \\
&\textbf{q} &\geq 0.
\end{array}
\end{equation*} 

Using LP (\ref{Uprice}) improves solvability and memory requirements in two major ways: when $k=2$, LP (\ref{Uprice}) requires just $|P_1| \cdot |P_2|$ variables, a tremendous reduction from $\prod_{i=1}^n |P_i|$. The number of variables {\em does not depend on $n$}. When the input measures have support sets of equal size $|P|$, this eliminates $|P|^{n-2}$ variables.  Additionally, the constraint matrix is stored in memory, so the benefit of replacing $A_p$ with $U_p$ is significant. 

Using LP (\ref{Uprice}) instead of LP (\ref{LPprice}) requires   additional preprocessing each iteration to construct the best-cost vector $b$, which does use the exponential-sized vector $(c^T-\y^T A_m)$. The preprocessing for LP (\ref{Uprice}), repeated each iteration, is given in Algorithm \ref{alg:PriceSetup}. In particular, Algorithm \ref{alg:PriceSetup} highlights the important selection of which indices $h$ correspond to the unique columns used in LP (\ref{Uprice}).  

\begin{algorithm}[h]
\caption{Setup of LP (\ref{Uprice})}\label{alg:PriceSetup}
\begin{algorithmic}
\Statex Intialize the vector of indices I of length $n_u$
\State Update $a = c^T-y_m^T A_m$
\For {$j=1,\ldots, n_u$}
\For {$h = 1+n_d \cdot (j-1), \ldots, n_d \cdot j$}
\If {$h = 1+n_d \cdot (j-1)$}
\State $b_j =a_h$
\State $I(j) = h$
\ElsIf {$a_h < b_j$}
\State $b_j = a_h$
\State $I(j) = h$
\EndIf
\EndFor 
\EndFor
\State Update objective of LP (\ref{Uprice})
\end{algorithmic}
\end{algorithm}

\subsection{Decomposition of Constraints for Exactly Two Measures}\label{sec:decomp} As in Example \ref{ex:A2}, we partition $A$ with $k=2$; that is, we always partition $A$ where $A_p$, and subsequently the matrix of unique columns $U_p$, contains all rows of constraints associated with exactly two measures. LP (\ref{Uprice}) has a linear objective function $b^T \textbf{q} + \sigma$; because of the linear objective and the structure of the constraints for two measures, LP (\ref{Uprice}) is a classical transportation problem \cite{ff-56b,m-16}, a special case of a minimum-cost flow problem. Therefore, LP (\ref{Uprice}) can be solved in strongly polynomial time \cite{abm-16,bp-18}.  

\begin{theorem} Let $U_p \textbf{q} = d_p$ be the constraints associated with exactly two measures. Then LP (\ref{Uprice})  is a classical transportation problem and can be solved in strongly polynomial time. \end{theorem}

We conclude this discussion with an examination of the efficiency of adding a column produced by LP (\ref{Uprice}) to LP (\ref{LPRM}). Once the column generation process has begun, the previous pricing problem LP (\ref{Uprice}) produces a solution $\textbf{q}$ containing the nonzero elements of a new $\p_J$ to be introduced to LP (\ref{LPRM}). This $\textbf{q}$ has, trivially, at most $|P_1| \cdot |P_2|$ nonzero elements, and in fact, there must exist a smaller solution of size $|P_1| + |P_2|-1$. Recall from Section \ref{sec:price} that $A_m$ is easily generated, so the pricing problem does not require $A_m$ to be stored in memory. The restricted master problem also does not require $A_m$ to be stored; instead, Algorithm \ref{alg:ACol} is used to calculate the new column $A_m \p_J$. Combined with the small number of nonzero elements of $\p_J$, a computation of $A_m \p_J$, as well as of $c^T \p_J$, can be done efficiently. For additional efficiency, the solver for LP (\ref{LPRM}) uses the previous solution as a warm start. Using the primal simplex method then typically finds a new optimal solution in just a few simplex steps for each update of LP (\ref{LPRM}). 

Next, we turn to the master problem and describe a method for generating an initial feasible start for both column generation algorithms. 

\section{Constructing a Feasible Solution}\label{sec:master}

To initialize column generation, both algorithms require enough variables such that an initial feasible solution exists, along with a feasible solution.

For LP (\ref{LPw}), a feasible solution is any $\w$ which solves the full system $A \w = d$, and a feasible master problem is produced by the variables associated with a positive value in $\w$. A feasible solution for LP (\ref{LPRM}) is related to the feasible solution $\w$ as follows. First, consider the introduction of a single initial column (so $J$ begins at $1$). Then the convexity constraint requires $\mu_1 = 1$, and the remaining constraints subsequently require $A_m \p_1 = d_m$. 
Furthermore, all $\p$ need to satisfy $A_p \p = d_p$, and thus $\p_1$ should also be a solution to the full system $A \w = d$.

We considered two methods for constructing a vertex: a greedy construction and the 2-approximation algorithm from \cite{b-17}. The appeal of the 2-approximation algorithm lies in its ability to efficiently provide a {\em good} potential optimum. However, we found in experiments that initialization with a 2-approximation vertex is consistently outperformed by initialization with a greedily constructed vertex (this was not due to increased time to generate the vertex but because additional iterations were required before strictly improving columns were found). Therefore, we present just the greedy construction algorithm, a generalization of the north-west corner rule.

The following algorithm greedily constructs a solution to $A \w = d$. The process begins by generating a combination $s_h = (\x_1^h, \x_2^h, \ldots, \x_n^h) \in S^*$. In the first step, each $\x_i^h$ has corresponding mass $d_i^h$, and the maximum mass that can be assigned to $s_h$ without violating the non-mass-splitting property is the minimum mass among $d_1^h, \ldots, d_n^h$. That minimum mass is placed at index $h$ in $w$, denoting that mass $w_h$ is transported to $\x_1^h, \ldots, \x_n^h$ from the optimal location for such mass assignment (the corresponding weighted mean $\x^h$). The algorithm than computes the remaining mass $\mathbb{d}$ each of the points $\x_1^h, \ldots, \x_n^h$ still needs to receive full mass $d_i^h$. Then the combination is updated; for each measure, if the current support point has not yet received full mass ($\mathbb{d} > 0$), the support point remains in the combination. However, at least one measure's support point has been fully supplied by the greedy mass assignment; for these measures a new support point is chosen, guaranteeing a new combination. The process then repeats, assigning the minimum mass not yet received at each support point to new combinations until all mass has been supplied.

This process is given in Algorithm \ref{alg:GreedyFeas}. Note that in the first step of each repeat of the algorithm, a combination of support points is formed before the corresponding index $h$. Therefore Algorithm \ref{alg:GreedyFeas} uses the double indexed notation $\x_{ij_i}$ as the $j_i^{th}$ support point in measure $P_i$ with corresponding mass $d_{ij_i}$, and computes the index $h$ for the combination. 

\begin{algorithm}[h]
\caption{Greedy Construction of $\w$: $A\w = d$ }\label{alg:GreedyFeas}
\begin{algorithmic}[1]
\Statex Input: vector $n_o$ containing number of consecutive ones for each $i$
\Statex Output: vector $\w$
\Statex For each $P_i$, and for $j_i = 1, \ldots, |P_i|$, initialize $\mathbb{d}_{ij_i}= d_{ij_i}$ 
\Statex Let $L = 1$, $m_1 = 0$, $\w = \textbf{0}$, and $j_i = 1 \ \  \forall i = 1, \ldots, n$
\While {$\sum_{l=1}^L m_{l} < 1$}
\State $m_{L} = \min\{\mathbb{d}_{ij_i}\}$ 
\State $h = \sum_{i=1}^n ((j_i-1)n_o(i))$
\State $w_h = m_L$
\For {$i=1, \ldots, n$}
\State $\mathbb{d}_{ij_i} = \mathbb{d}_{ij_i}-m_{L}$
\If {$\mathbb{d}_{ij_i} = 0$}
\State $j_i = j_i +1$
\EndIf
\EndFor
\State $L = L+1$
\EndWhile
\end{algorithmic}
\end{algorithm}

\begin{theorem}\label{thm:GreedyEff} Let $P_1, \ldots, P_n$ be discrete probability measures. Then Algorithm \ref{alg:GreedyFeas} runs in $\mathcal{O}(n\sum_{i=1}^n |P_i| )$ in the arithmetic model of computation. \end{theorem}

\begin{proof} 
First, we show that the number of nonzero elements produced by Algorithm \ref{alg:GreedyFeas}, which is also the number of repetitions of the outer loop of Algorithm \ref{alg:GreedyFeas}, is between $\max_{1 \leq i \leq n} \{|P_i|\}$ and  $\sum_{i=1}^n |P_i| -n+1.$ The lower bound, $\max_{1 \leq i \leq n} \{|P_i|\}$, is an immediate consequence of the non-mass-splitting property maintained by Algorithm \ref{alg:GreedyFeas}. For the upper bound, $\sum_{i=1}^n |P_i| -n+1$, note that the last iteration must fully supply the mass to $n$ points, one $\x_{i}$ from all $P_i$, because the total mass for each $P_i$ is the same (one). In each previous iteration, the minimum number of support points whose index $j_i$ changes is one, for a total of $\sum_{i=1}^n (|P_i|-1) +1 = \sum_{i=1}^n |P_i| -n+1$ iterations. 

Thus the outer loop runs in $\sum_{i=1}^n |P_i|$ time. Since each step inside the loop of Algorithm \ref{alg:GreedyFeas} requires at most linear-in-$n$ elementary operations, we obtain Theorem \ref{thm:GreedyEff}. \qed \end{proof}

\begin{corollary} For $n$ probability measures with support sets of size at most $|P|$, Algorithm \ref{alg:GreedyFeas} runs in $\mathcal{O}(n^2)$ in the arithmetic model of computation. \end{corollary}

\begin{proof}
Let $P_1, \ldots, P_n$ be discrete probability measures with a bound $|P|$ on the size of their support sets. Then $\sum_{i=1}^n |P_i| \leq n |P|$, and $\mathcal{O}(n\sum_{i=1}^n |P_i|  )$ becomes $\mathcal{O}(n^2 ).$ \qed 
\end{proof}

As an additional consequence of the iteration bound $\sum_{i=1}^n |P_i| -n+1$, the number of nonzero mass elements of $\w$ are bounded. Therefore the setup of the master problem LP (\ref{LPRM}) using a result produced by Algorithm \ref{alg:GreedyFeas} is efficient.


We now show that Algorithm \ref{alg:GreedyFeas} produces a vertex of the polytope generated by the constraints $A\w=d$. 

\begin{theorem} Algorithm \ref{alg:GreedyFeas} generates a vertex of the polytope $\{\w \in \R^d: A\w = d, \w \geq 0\}$. \end{theorem}

\begin{proof}
Let $A$, $d$ be given and let $\w$ be generated using Algorithm \ref{alg:GreedyFeas}. We show there exists a $c$ such that $\w$ is the unique optimal solution to:
\begin{equation*}
\begin{array}{crl}
\mathrm{min}&  c^T \w & \nonumber \\
\mathrm{s.t.  } &A \w = &d \\
&\w \geq & 0. \nonumber
\end{array}
\end{equation*}
Let $M$ be the set of nonzero elements of $\w$, with size $|M| = L$. Order the elements of $M$ in order of construction by Algorithm \ref{alg:GreedyFeas}, $m_1, \ldots, m_L$. Also order the associated indices $h_1, \ldots, h_L$ as calculated by Algorithm \ref{alg:GreedyFeas}. 

First, we show that $\w$ is a feasible solution to the above system. For each support point $\x_{j_i}$ in each measure $P_i$, the current value $\mathbb{d}_{ij_i}$ is initialized as its full mass $\mathbb{d}_{ij_i}= d_{ij_i}$. The algorithm begins with a combination $h$ of support points from each measure, identifies the smallest mass $m_L$ among them (line $2$) and sets the mass for this combination $w_h$ to $w_h=m_L$ (line $3$ to get the correct index of the combination; line $4$ for the assignment). Then the current masses $\mathbb{d}_{ij_i}$ of all support points in the combination are reduced by $m_L$ (line $6$). The current mass of at least one of the support points must have dropped to $0$; then a new support point is picked from the respective measure (lines $7$ and $8$) and the process is repeated. To see why this yields a feasible solution, recall that the total mass in each measure $P_i$ is precisely $1$ and note that, in line $6$, the total current mass in each measure is dropped by the same value $m_L$. The algorithm runs until $\sum_{l=1}^L m_{l} = 1$ (line $1$), i.e., until the total mass of each support point in each measure is fully accounted for. This gives feasibility of $\w$.

Next, we construct a $c$ such that $\w$ is a unique optimal solution. Let $c_{h_1} = 1$, $c_{h_2} = 2$, \ldots, and $c_{h_L} = L$. Let all other $c_h$, those whose $h$-index is not in $h_1, \ldots, h_L$, be $\sum_{i=1}^n |P_i| -n+2$ (Recall: $|M|  \leq \sum_{i=1}^n |P_i| -n+1$). 

By construction, removing mass from a combination with a lower index and assigning it to a combination with higher index in $M$, that is, from $m_j$ to $m_k$ with $j \leq k$, will strictly increase the value of $c^T \w$. This includes moving mass to a combination with no mass in $\w$, that is, with an index not in $M$. 

So it suffices to show that mass cannot be reassigned from $m_k$ to $m_j$, $j \leq k$. The mass $m_j$ is chosen such that for at least one $x_{j_i}$, the mass $d_{j_i}$ has been fully supplied. Therefore $m_j$ cannot be increased without violating the constraints $A \w = d$.
  
Therefore $\w$ minimizes $c^T \w$ subject to $A \w = d$, since the maximum mass allowable is assigned to the cheapest costs. Furthermore, $\w$ does so uniquely, since any change in its elements will strictly increase the value of $c^T \w$ due to the construction of $c$. Therefore $\w$ is a vertex. \qed 

\end{proof} 

%
\begin{figure}[t]
\begin{center}
\fbox{\includegraphics[scale=.45]{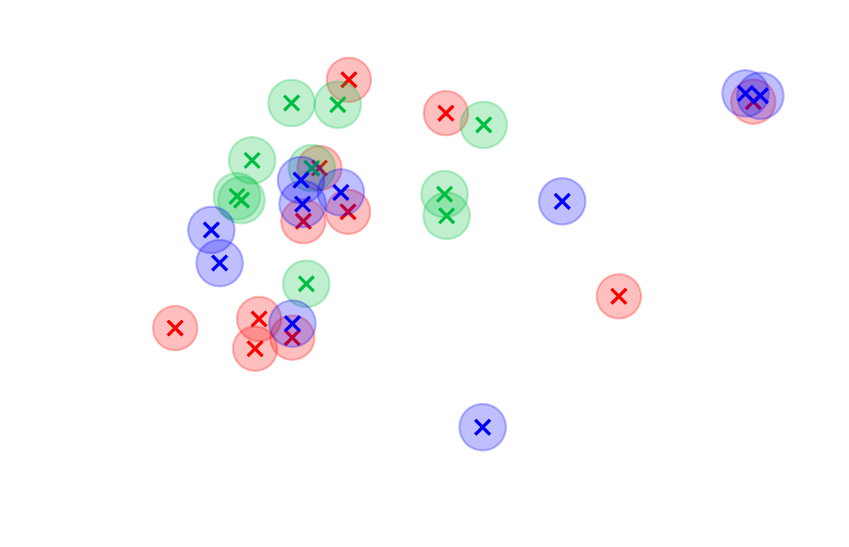}}
\fbox{\includegraphics[scale=.45]{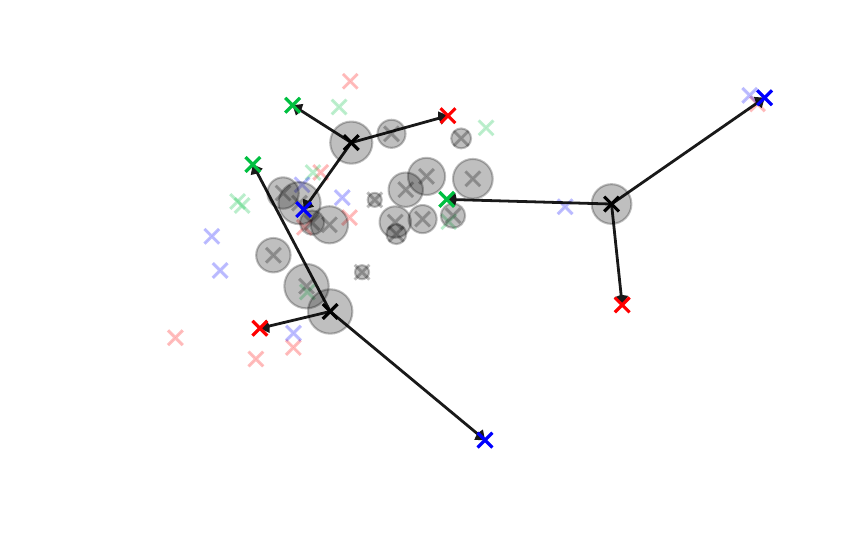}}
\end{center}
\caption{(left)  Three measures in general position with 10 or 11 support points and equally distributed mass. (right) A greedily constructed feasible solution. Transport from three sample points -- those constructed first, fifth, and seventeenth -- is shown (arrows). Each support point is the weighted mean of its three destination points. }\label{fig:GenPos}
\end{figure}

In Figure \ref{fig:GenPos} (left), we display an example with three measures, two with 10 support points and one with 11 support points. Each measure has equally distributed mass. Applying Algorithm \ref{alg:GreedyFeas} results in a feasible solution supported on 20 weighted means of varying mass, displayed in Figure \ref{fig:GenPos} (right), along with the transport for three sample points.

\section{Computations}\label{sec:comp}

The primary goal of these experiments is to demonstrate, for general position measures, the computational benefits of column generation algorithms over the full linear program. To this end, we construct measures from a real-world data set containing event locations given in longitude and latitude. Because the events occur without known structure, probability measures with these support points are in general position. The generated measures have varying numbers of support points with uniformly distributed mass, and the weight of each measure is inversely proportional to the number of support points. All computations have been run on a laptop (MacBook Pro, 2.4 GHz Intel Core i9, 32 GB of RAM, SSD). Data processing and the setup of the LPs were implemented in C\texttt{++} and the LPs were solved using Gurobi 8.0. The source code is available at \url{https://github.com/StephanPatterson/Barycenter-Formulations}. For a meaningful comparison, we set Gurobi to run without presolvers and using the same algorithm (primal simplex method) in all experiments.

We want comparisons to exact computations, which as previously discussed, are hard \cite{ab-21,bp-21a}. Even when the measures contain a small number of support points, LP (\ref{LPw}) may contain millions of variables. Therefore, the following analysis focuses primarily on measures with small support sets (2-12 support points per measure); the improved scaling on our column generation algorithms would allow for measures of more moderate size, but not orders of magnitude larger. Throughout this section, we use the number of variables in LP (\ref{LPw}) as a reference label for a particular instance. 

The second goal of these experiments is to examine the practical behavior of variations in implementation. To this end, we compare three variants of column generation applied directly to LP (\ref{LPw}) and two versions using Dantzig-Wolfe decomposition. The two variants for the Dantzig-Wolfe reformulation differ only in the choice of which two measures are moved to the pricing problem. In the ``DW-L'' variant, the two measures with the largest number of support points are moved to the pricing problem, while the ``DW-A'' variant makes an arbitrary choice of two measures. 

The three variants for column generation directly on LP (\ref{LPw}) vary on the number of columns introduced per iteration; ``1-col'' refers to the standard column generation strategy of introducing the variable with the greatest reduced cost, thus introducing one variable per iteration. We have also included the strategy introducing all variables with improved cost, labeled ``all-col'', and a heuristic compromise between the strategies, introducing the best $n$ columns per iteration, labeled ``$n$-col''. We also considered, but ultimately discarded, a variant introducing the first $n$ columns each iteration; while this has the benefit of avoiding the processing of the full exponential-sized cost vector each iteration, several times more variables were introduced, leading to slower solution times and larger problem sizes than the best $n$ variant in all but one of our experiments. We believe this is due to the highly structured nature of $A$.

\begin{table}
\scriptsize
\begin{center}
\begin{tabular}{|c|c|c|c|c|c|c|c|c|c|c|c|c|}\hline
 &\multicolumn{2}{|c|}{ LP (\ref{LPw})} & \multicolumn{10}{|c|}{Column Generation } \\ \hline
&\multicolumn{2}{|c|}{ }& \multicolumn{2}{|c|}{1-col} &  \multicolumn{2}{|c|}{$n$-col } & \multicolumn{2}{|c|}{all-col}&\multicolumn{2}{|c|}{ DW-L}&\multicolumn{2}{|c|}{DW-A} \\ \hline
n & Var  &Time   & Var & Time & Var & Time & Var & Time & Var & Time & Var & Time \\ \hline
12 & 2,177,280 & \textbf{5.22} & 270 &22.96 & 825 & 7.01 & 1,199,800 & 10.60 & 646 & 18.78&478&14.67 \\ \hline
14 & 4,976,640 & 13.68 & 234& 50.42  & 634 & \textbf{11.73} & 2,688,032 & 35.39 & 581 & 39.70&638&41.35  \\ \hline
14 & 5,971,968 & 56.68 & 232 & 60.31 & 669 & \textbf{14.51} & 3,272,679 & 29.11 & 684 & 56.50 & 422 & 49.91 \\ \hline
12 & 25,288,704 & 235.17 & 322 & 354.28 & 903 & \textbf{90.70} & 12,727,161& 344.33 & 489 & 162.41 & 625 & 218.18 \\ \hline
14 & 28,449,792 & 358.84 & 308 & 423.30 & 1,004 & \textbf{109.97} & 14,572,552 & 215.50 & 803 & 302.01& 545 & 252.43 \\ \hline
15 & 31,850,496 & 378.90 &  309 & 480.10 & 961 & \textbf{106.10} & 17,992,481 & 258.98 & 1,604 & 689.25 & 779 & 303.90 \\ \hline
17 & 63,700,992 & 1754.08 & 364 & 1,559.79 & 1,395 & \textbf{404.12}& 33,848,981 & 1,129.83 & 2,403 & 3,174.96 & 2,081 & 1,980.67 \\ \hline
17 & 84,934,656 & 1858.22 & 348 & 2,225.41 & 1,209& \textbf{364.50} & 44,303,632& 1,945.65 & 2,023 & 2,882.76 & 1,209 & 1,538.60 \\ \hline
18 & 127,401,984 & 3588.32 & 386 & 3,145.88 & 1,189& \textbf{699.37} & 62,383,222 & 2,467.41 & 2,121 & 4,402.53 & 2,924 & 5,105.93\\ \hline
17 & 148,635,648 & * & 341 & 3,498.02 & 1,155 & \textbf{844.44} & 83,870,587& 4,001.37 & 1,119 & 3,204.85 & 724 & 1,621.57 \\ \hline
18 & 191,102,976 & * & 364 & 3,507.80 & 1,376& \textbf{807.66} & 100,681,949& 10,015.66 & 2,019 & 5,588.05 & 899 & 4,477.56 \\ \hline
\end{tabular}
\end{center}
\caption{Comparison of column generation algorithms for $n$ measures per experiment, including the number of variables (Var) introduced by each algorithm. Times, including setup, are given in seconds with fastest times in bold.  Each measure has a small number (between 2 and 12) of points in general position. For larger instances, a direct solution was not possible due to memory limitations (*).}\label{tab:initcompare}
\end{table}

The total running times for these experiments are shown in Table \ref{tab:initcompare}. All of the column generation algorithms are able to find solutions to experiments for which LP (\ref{LPw}) is too large for the laptop (*). The classic column generation algorithm, 1-col, typically does not show an improvement in solving speed over a direct computation using LP (\ref{LPw}) for these experiments, which was our motivation for considering the other, heuristic strategies for introducing columns. The Dantzig-Wolfe reformulation algorithms, DW-L and DW-A, usually show minor improvements over LP (\ref{LPw}), though one variant does not reliably outperform the other. 
The fastest run times consistently come from the approach that introduces the best $n$ columns per iteration. The all-columns approach also typically outperforms a direct solve; the size of the problem is approximately half of the full linear program and the algorithm completes in a handful of iterations (at most 4). 

All column generation algorithms dramatically reduce the number of variables introduced, which results in significantly lower memory requirements. The maximum memory used during the execution of each experiment is shown in Table \ref{tab:memcomp}; the Dantzig-Wolfe reformulation is the most memory efficient algorithm  due to its compact restricted master problem LP (\ref{LPRM}) and condensed pricing problem LP (\ref{Uprice}).

\begin{table}[]
    \centering
    \begin{tabular}{|c|c|c|c|c|c|c|} \hline
         & LP (\ref{LPw}) &1-col & $n$-col & all-col & DW-L & DW-A  \\ \hline
        2,177,280 &2,380& 42 & 44 & 1,330 & 33 & \textbf{30} \\ \hline
        4,976,640 &6,080& 85 & 86 & 1,530 & 63 & \textbf{56} \\ \hline
        5,971,968 &7,360& 100 & 102 & 2,820 & \textbf{70} & 72 \\ \hline
        25,288,704 &28,210& 398 & 399 & 12,690 & \textbf{212} & 223 \\ \hline
        28,449,792 &36,040& 447 & 450 & 10,720 & \textbf{244} & 267 \\ \hline
        31,850,496 &42,820& 499 & 501 & 20,790 & 286 & \textbf{274}\\ \hline
        63,700,992 & 94,310 & 990 & 995 & 28,810 & \textbf{542} & 564\\ \hline
        84,934,656 & 125,740 & 1,290 & 1,290 & 57,220 & \textbf{691} & 751 \\ \hline
        127,401,984 & 199,450 & 1,920 & 1,930 & 84,130 & \textbf{1,020} & 1,060\\ \hline
        148,635,648 & * & 2,240 &  2,250 & 67,200 & \textbf{1,150} & 1,260 \\ \hline
        191,102,976 & * & 2,880 & 2,880 & 86,810 & \textbf{1,520} & 1,920 \\ \hline
    \end{tabular}
    \caption{Maximum memory used by column generation algorithms for each instance from Table \ref{tab:initcompare}, given in MB. }
    \label{tab:memcomp}
\end{table}

Since the fastest running times came from column generation on LP (\ref{LPw}), but the best memory efficiency from the Dantzig-Wolfe reformulation, we re-ran the experiments with columns deletion -- the removal of columns in the master problem when they leave the basis during the simplex method -- to see if the memory requirements of column generation on LP (\ref{LPw}) could be further reduced. The results of these experiments is given in Table \ref{tab:coldelete}; however, column deletion resulted in a very minor reduction in maximum memory usage while dramatically increasing running times. The memory reduction was not sufficient for the algorithms on LP (\ref{LPw}) to be as efficient as a Dantzig-Wolfe implementation. 

\begin{table}
\centering
\begin{tabular}{|c|c|c|c|c|c|} \hline
    & \multicolumn{2}{|c|}{1-col} & \multicolumn{2}{|c|}{$n$-col} & DW-L \\ \hline
    & None & With & None & With & None  \\ \hline
    2,177,280 & 42 & 41& 44  & 42 & 33 \\ \hline
    4,976,640 & 85  & 84& 86 & 84 & 63 \\ \hline
    5,971,968 & 100  & 100& 102 & 100 &70\\ \hline
    25,288,704 & 398  & 398& 399 & 399 &212 \\ \hline
    28,449,792 & 447 & 446 & 450 & 446 &244 \\ \hline
    31,850,496 & 499 & 499 & 501 & 500 & 286\\ \hline
    63,700,992 & 990 &988 & 995 & 989 & 542\\ \hline
    84,934,656 & 1,290 &1,280 & 1,290 & 1,280 &691 \\ \hline
    127,401,984 & 1,920 & 1,920 & 1,930 & 1,920 &1,020\\ \hline
    148,635,648 & 2,240 & 2,240 & 2,250 & 2,240 & 1,150\\ \hline
    191,102,976 & 2,880 & 2,880 & 2,870 & 2,870 & 1,520 \\ \hline
\end{tabular}
\caption{Maximum memory used by the 1-col and $n$-col column generation algorithms, without deletion (None) and with column deletion (With), given in MB. The reduction in memory requirements is negligible. The memory use for DW-L is repeated from Table \ref{tab:memcomp} for comparison.} \label{tab:coldelete}
\end{table}

In the column generation algorithms on LP (\ref{LPw}), the bottleneck for faster running times is the explicit choosing of new columns, which is dependent on the exponential-sized cost vector $c$. The efficiency of each step of the Dantzig-Wolfe reformulation algorithm is somewhat less apparent; we examine the breakdown of running times each iteration in Table \ref{tab:Percentage}. The processing of $c$ to produce the updated, unique best-cost vector $b$ for the pricing problem is the majority of computational effort, while solving the pricing problem and subsequent master problem are efficient.

\begin{table}[t]
\begin{center}
\begin{tabular}{|c|c|} \hline
Step & Percentage of Computation Time \\ \hline
Setup LP (\ref{LPRM}) & $< 0.1 \%$ \\ \hline
Solve LP (\ref{LPRM}) & 1.1\% \\ \hline
Update $(c^T-\y^TA_m)$ & 72.5\% \\ \hline
Calculate $b$ & 26.2\% \\ \hline
Solve LP (\ref{Uprice})& $< 0.1\%$\\ \hline
\end{tabular}
\end{center}
\caption{Percentage computation time per step of an average iteration of column generation. Most of the effort is spent on the setup of LP (\ref{Uprice}); the computation times for solving LP (\ref{Uprice}) and the setup of the next LP (\ref{LPRM}) contribute negligibly to the total.}\label{tab:Percentage}
\end{table}

\section{Concluding Remarks}\label{sec:conc}

The computation of an exact barycenter is costly in practice, and provably hard for data in general position \cite{ab-21,bp-21a}. In this paper, we studied two column generation strategies - one on a suitable linear programming formulation for such data, one based on a Dantzig-Wolfe reformulation. While both of these provide significant improvements in scalability, especially though a memory-efficient implementation, computations remain hard. In this work, we used a couple of standard column generation techniques to improve the practical performance, such as the generation of multiple columns in each iteration, simple deletion strategies, or the generation of any (not necessarily best) improving columns. They typically have a positive impact, and we believe further refinements of the presented approach are interesting direction of future work, but they cannot overcome the underlying hardness of the problem. 

As most barycenter algorithms require an explicit specification of a set of possible support points, and the size of this set is a bottleneck to computations, the {\em direct and efficient} generation of support points remains a key interest in the community. It translates to an efficient generation of columns for LP (\ref{LPw}). It remains open whether it is possible to efficiently generate a (single) improving column; hardness of an exact barycenter computation implies that either the generation of a column itself or the number of columns that have to be generated cannot be polynomial.

The methods to do so will require a quite different approach: while we showed that it is efficient to evaluate the reduced cost for any given combination $s_h \in S^*$, the challenge lies in finding an improving one without an explicit evaluation of each combination in $S^*$. We see potential for a competitive algorithm through an approximation of the data going into the reduced cost vector computation, which may lead to a heuristic algorithm, or through the setup and solution of an integer program for pricing, which may lead to further improvements for an exact computation.

\section*{Acknowledgments}

We would like to thank Ethan Anderes for the implementation of a visualization basis for barycenters used in \cite{abm-16}, which we modified to produce the figures of Sections \ref{sec:intro} and  \ref{sec:master}. We would also like to thank Jon Lee for the many helpful discussions about transportation problems and total unimodularity.

The authors gratefully acknowledge support of this work by the National Science Foundation, Algorithmic Foundations, Division of Computing and Communication Foundations, under grant 2006183 {\em Circuit Walks in Optimization}; by the Airforce Office of Scientific Research under grant FA9550-21-1-0233 {\em The Hirsch Conjecture for Totally-Unimodular Polyhedra}; and by the Simons Foundation under Collaboration Grant 524210 {\em Polyhedral Theory in Data Analytics} before.

\bibliography{barycenters_literature}
\bibliographystyle{plain}

\end{document}